\documentclass[11pt]{article}
\usepackage{a4wide}
 \usepackage{theoremref}
 \usepackage{amssymb,amsmath,amsfonts,amsthm}
 \usepackage{bbm}
 \usepackage{graphicx,caption,color}
 \usepackage{mathabx}

\usepackage[utf8]{inputenc}
\usepackage{relsize}
 \usepackage{enumitem}
 
\newtheorem{theorem}{Theorem}[section]
\newtheorem{lemma}[theorem]{Lemma}
\newtheorem{claim}[theorem]{Claim}
\newtheorem{corollary}[theorem]{Corollary}
\newtheorem{remark}[theorem]{Remark}

\newtheorem{question}[theorem]{Question}
\newtheorem{problem}[theorem]{Problem}

\newtheorem{observation}[theorem]{Observation}

   \newcommand{\lab}[1]{\label{#1}}                
   \newcommand{\thlab}[1]{\thlabel{#1}} 

\newcommand{\cM}{{\mathcal M}}

\newcommand{\cT}{{\mathcal T}}

\newcommand{\tentacle}{has Property~$\cT$}

\newcommand{\T}{$\cT$}

\renewcommand{\phi}{\varphi}

\newcommand{\dist}{\operatorname{dist}}

\renewcommand\tilde[1]{\widetilde{#1}}

\newcommand\se{\subseteq}
\newcommand\nse{\nsubseteq}  	
\newcommand\sne{\varsubsetneq} 	
\newcommand\sm{\setminus}
\renewcommand\to{\rightarrow}
\newcommand\nto{\nrightarrow}

\begin{document}

\title{On minimal Ramsey graphs and Ramsey equivalence in multiple colours}

\author{
\quad{Dennis Clemens
\thanks{Technische Universit\"at Hamburg, Institut f\"ur Mathematik, Germany
Email: dennis.clemens@tuhh.de }} 
\quad{Anita Liebenau
\thanks{School of Mathematics and Statistics, UNSW Sydney,
Sydney NSW 2052, Australia. 
Email: a.liebenau@unsw.edu.au. Supported by an ARC Decra Fellowship. Previously at Monash University where this research was partly carried out.}} 
\quad{Damian Reding  
\thanks{Technische Universit\"at Hamburg, Institut f\"ur Mathematik, Germany
Email: damian.reding@tuhh.de}}
}

\maketitle

\begin{abstract}
For an integer $q\ge 2$, a graph $G$ is called $q$-Ramsey for a graph $H$ if every $q$-colouring of the edges of $G$ contains a monochromatic copy of
$H$. If $G$ is $q$-Ramsey for $H$, yet no proper subgraph of $G$ has this property then $G$ is called $q$-Ramsey-minimal for $H$. 
Generalising a statement by Burr, Ne\v{s}et\v ril and R\"odl from 1977 we prove that, for $q\ge 3$,  
if $G$ is a graph that is not $q$-Ramsey for some graph $H$ then $G$ is contained as an induced subgraph in an infinite number of $q$-Ramsey-minimal graphs for $H$, as long as $H$ is $3$-connected or isomorphic to the triangle. For such $H$, the following are some consequences. 
\begin{itemize}

\item
For $2\le r< q$, every $r$-Ramsey-minimal graph for $H$ is contained as an induced subgraph in an infinite number of $q$-Ramsey-minimal graphs for $H$. 

\item 
For every $q\ge 3$, there are $q$-Ramsey-minimal graphs for $H$ of arbitrarily large maximum degree, genus, and chromatic number. 

\item
The collection $\{\cM_q(H) : H \text{ is 3-connected or } K_3\}$ forms
an antichain with respect to the subset relation, 
where $\cM_q(H)$ denotes the set of all graphs that are
$q$-Ramsey-minimal for $H$. 
\end{itemize}

We also address the question which pairs of graphs satisfy $\cM_q(H_1)=\cM_q(H_2)$, in which case $H_1$ and $H_2$ are called $q$-equivalent. We show that two graphs $H_1$ and $H_2$ are $q$-equivalent for even $q$ if they are $2$-equivalent, and that in general $q$-equivalence for some $q\ge 3$ does not necessarily imply $2$-equivalence. Finally we indicate that for connected graphs this implication may hold: Results by Ne\v{s}et\v ril and R\"odl and by Fox, Grinshpun, Liebenau, Person and Szab\'o imply that the complete graph is not $2$-equivalent to any other connected graph. We prove that this is the case for an arbitrary number of colours. 

\end{abstract} 

\section{Introduction}
A graph $G$ is $q$-Ramsey for $H$, denoted by $G\to (H)_q$, if every $q$-colouring of the edges of $G$ contains a monochromatic copy of $H$. Many interesting questions arise when we consider those graphs $G$ which are minimal with respect to $G\to (H)_q$. A graph $G$ is $q$-Ramsey-minimal for $H$ (or $q$-minimal for $H$) if $G\to (H)_q$ and $G'\nto(H)_q$ for every proper subgraph $G'\sne G$. We denote the family of such graphs by $\mathcal{M}_q(H)$. The fact that $\mathcal{M}_q(H)\neq\varnothing$ for every graph $H$ and every integer $q\ge 2$ is a consequence of Ramsey's theorem \cite{r1930}. 
Burr, Erd\H{o}s, and Lov\'asz~\cite{bel1976} initiated the study of properties of graphs in $\cM_2(K_k)$ in 1976, where as usual $K_k$ denotes the complete graph on $k$ vertices.  Their seminal paper raised numerous questions on minimal Ramsey graphs that were addressed by various mathematicians in subsequent years~\cite{befs1981,bnr1984,l1994,bdks2001,rs2008}. 

Various graph parameters have been studied for graphs in $\mathcal{M}_q(H)$, the most prominent being the Ramsey number $r_q(H)$ which is the smallest number of vertices of a graph in $\cM_q(H)$. 
When $H$ is the complete graph we also write $R_q(k)$ for $r_q(K_{k})$. 
Estimating $R_q(k)$ or even $R_2(k)$ is one of the fundamental open problems in Ramsey theory. It is known that $2^{k/2+o(k)}\leq R_2(k)\leq 2^{2k-o(k)}$ where the best lower bound is due to Spencer~\cite{s1977} improving a result by Erd\H{o}s~\cite{e1947}, and the best known upper bound is due to Conlon~\cite{c2009}, improving earlier bounds by Erd\H{o}s and Szekeres~\cite{es1935}, R\"odl~\cite{gr1987}, and Thomason~\cite{t1988}.  
Quite surprisingly, some other parameters could be determined precisely. Ne\v set\v ril and R\"odl~\cite{nr1976} showed, for example, that the smallest clique number of a graph in $\cM_q(H)$ is exactly the clique number of $H$, extending earlier work by Folkman~\cite{f1970}. Furthermore, the smallest chromatic number and the smallest connectivity of a graph in $\cM_q(H)$ are known for all $H$ and $q\ge 2$, see~\cite{bel1976} and~\cite{bnr1984}. 
A parameter of ongoing interest is $s_q(H)$, the smallest minimum degree of a graph $G\in \cM_q(H)$. The value of $s_2(H)$ is known for some graphs $H$, including cliques~\cite{bel1976}, complete bipartite graphs~\cite{fl2006}, trees and cycles~\cite{szz2010}, and complete graphs with a pendant edge~\cite{fglps2014}. The asymptotic behaviour of $s_q(K_k)$ was considered when $q\to\infty$ in~\cite{fglps2016,gw2017}, and when $k\to \infty$ in~\cite{hrs2018}.

In this paper we are interested in the interplay between $\cM_q(H)$ and $\cM_r(H')$ when $q\neq r$ or when $H$ and $H'$ are nonisomorphic. 
Clearly, every graph $G$ that is a $q$-minimal graph for some graph $H$ is $r$-Ramsey for $H$, for all $2\le r\le q$, and thus contains an $r$-minimal graph as an induced subgraph. Our first contribution complements this observation in the sense that every $r$-minimal graph $G$ can be obtained this way from a $q$-minimal graph $G'$, as long as $H$ satisfies some connectivity conditions. Following standard notation we write $H\cong H'$ if $H$ and $H'$ are isomorphic. 
\begin{theorem}\thlab{rodlscorollary} 
Let $H$ be a $3$-connected graph or $H\cong K_3$ and let $q>r \geq 2$ be integers. 
Then for every $F\in \cM_r(H)$ 
there are infinitely many graphs $G\in \cM_q(H)$ such that $F$ is an induced subgraph of $G$. 
\end{theorem}
In fact, this result is an immediate consequence of the following more general statement.
\begin{theorem}\thlab{thm:containment}
Let $H$ be a $3$-connected graph or $H\cong K_3$, let $q\geq 2$ be an integer and 
let $F$ be a graph which is not $q$-Ramsey for $H$. 
Then there are infinitely many graphs $G\in \cM_q(H)$ such that $F$ is an induced subgraph of $G$. 
\end{theorem}

For the assertions of \thref{rodlscorollary,thm:containment} to hold it is clearly necessary that $H$ is {\em Ramsey infinite}, that is $\cM_q(H)$ is infinite. Some graphs including, for example, star forests with an odd number of edges, are known not to be Ramsey infinite. Faudree~\cite{f1991} provided a full characterization of forests that are Ramsey infinite. Furthermore, it follows from~\cite[Corollary 4]{rr1995} by R\"odl and Ruci\'nski that $H$ is Ramsey infinite if $H$ contains a cycle. It may well be possible that the assertions of \thref{rodlscorollary,thm:containment} hold for all graphs $H$ that are Ramsey infinite. 

The 2-colour version of \thref{thm:containment} was proved by Burr, Ne\v{s}et\v{r}il and R\"odl~\cite{bnr1984}, extending earlier work by Burr, Faudree and Schelp~\cite{BFS} who proved the statement for $q=2$ and when $H$ is a complete graph. Yet, it is this multi-colour version which implies \thref{rodlscorollary} as a corollary. 
As in \cite{bnr1984} for $q=2$, \thref{thm:containment} also implies the existence of multicolour Ramsey-minimal graphs with arbitrarily large maximum degree, genus and chromatic number. Indeed, it is well-known that, for a fixed graph $H$ containing a cycle and for a fixed integer $k$, the uniform random graph $G(n,p)$ does not contain $H$ as a subgraph and has maximum degree, genus and chromatic number at least $k$ with probability tending to 1 as $n\to\infty$, for some $p=\Theta(1/n)$. Take $F$ in \thref{thm:containment} to be such a graph drawn from $G(n,p)$.

Another implication of \thref{thm:containment} that we find noteworthy is the following.
\begin{corollary}\thlab{wow}
Let $H$ be a $3$-connected graph or $H\cong K_3$ and let $q\geq 2$ be an integer. 
Suppose that $\mathcal{M}_q(H)\se\mathcal{M}_q(H')$ for some arbitrary graph $H'$. 
Then $\mathcal{M}_q(H)=\mathcal{M}_q(H')$. 
\end{corollary}
We provide the short argument in Section~\ref{sec:containment}. 
Another way to view \thref{wow} is that if both $H$ and $H'$ are 3-connected or isomorphic to $K_3$ then the two sets $\mathcal{M}_q(H)$ and $\mathcal{M}_q(H')$ are either equal or incomparable with respect to the subset relation, i.e.~the set $\{\cM_q(H) : H \text{ is $3$-connected or } K_3\}$ forms an antichain with respect to the subset relation.
We find it instructive to note at this point that for such $H, H'$, in fact, $\mathcal{M}_q(H)=\mathcal{M}_q(H')$ is only possible if $H$ is isomorphic to $H'.$
\begin{theorem}\thlab{dennisObs}
Let $H$ and $H'$ be non-isomorphic graphs that are either 3-connected or isomorphic to $K_3$. Then $\mathcal{M}_q(H)\neq\mathcal{M}_q(H')$ for all $q\ge 2$. 
\end{theorem}

It is now natural to ask which pairs of graphs $H$ and $H'$ do satisfy $\cM_q(H)=\cM_q(H')$. For an integer $q\ge 2$ let us call two graphs $H$ and $H'$ {\em $q$-Ramsey equivalent} (or just {\em $q$-equivalent}) if $\cM_q(H)=\cM_q(H')$. 
The notion was introduced by Szab\'o, Zumstein and Z\"urcher~\cite{szz2010} in the case of two colours to capture the fact that $s_2(H)=s_2(H')$ for some graphs $H$ and $H'$ merely because $\cM_2(H)=\cM_2(H')$. 
We are particularly interested in the relationship between 2-colour equivalence and multi-colour equivalence, i.e.~what can we infer from known results for 2 colours to more colours? 

To briefly survey which pairs of graphs are known to be 2-equivalent, 
let $H+sH'$ denote the graph formed by the vertex disjoint union of a copy of $H$ and $s$ copies of $H'$, where we omit $s$ when $s=1$.
It is straight-forward to see that $K_k$ is 2-equivalent to $K_k+sK_1$ if and only if $s\le R(k)-k$, see e.g.~\cite{szz2010}. For $k\ge 4$, $K_k$ and $K_k+K_2$ are known to be $2$-equivalent. In fact, Szab\'o, Zumstein and Z\"urcher~\cite{szz2010} proved that for $2\le t\le k-2$ and $s< (R(k-t+1,k)-2(k-t))/2t$ the graphs $K_k$ and $K_k+sK_t$ are $2$-equivalent, where $R(k,\ell)$ denotes the smallest integer $n$ such that every red/blue-colouring of the edges of $K_n$ contains a red copy of $K_k$ or a blue copy of $K_{\ell}$. 
For the case $t=k-1$, Bloom and the second author~\cite{bl2018} show that $K_k$ and $K_k+K_{k-1}$ are 2-equivalent for all $k\ge 4$.  
(The requirement $k\ge 4$ is necessary in both~\cite{szz2010} and~\cite{bl2018}. Furthermore, the result in~\cite{szz2010} is optimal up to a factor of roughly 2, the result in~\cite{bl2018} is optimal in the sense that $K_k+K_{k-1}$ cannot be replaced by $K_k+2K_{k-1}$. We comment on these {\em non-equivalence} results further below.)
Axenovich, Rollin, and Ueckerdt~\cite{aru2015} provide a tool to lift these 2-equivalence results to $q$-equivalence.

\begin{theorem}[Theorem 10 in \cite{aru2015}] \thlab{aru}
If two graphs $H$ and $H'$ are $2$-equivalent and $H\se H'$ then $H$ and $H'$ are $q$-equivalent for every $q\ge 3$. 
\end{theorem}

In particular, the pairs $K_k$ and $K_k+s K_t$ are $q$-equivalent for every $q\ge 3$ whenever they are 2-equivalent. It would be desirable to remove the condition $H\se H'$ from \thref{aru}. 
In general, the following lifts $2$-equivalence (without the subgraph requirement) to $q$-equivalence for even $q$.

\begin{observation}\thlab{positiveLC}
Let $a,b,q,r$ be non-negative integers such that $q,r\ge 2$. If $H$ and $H'$ are $q$- and $r$-equivalent then they are $(aq+br)$-equivalent. 
\end{observation}
Indeed, the result follows by induction on $a+b\geq 1$ with the case $a+b=1$ given by assumption. Without loss of generality suppose that $H$ and $H'$ have been shown to be $n$-equivalent, where $n=(a-1)q+br$. Now suppose $G$ is a graph such that $G\to(H)_{n+q}$. We claim that then $G\to(H')_{n+q}$ as well. Fix an $(n+q)$-colouring $c:E(G)\to[n+q]$ of the edges of $G$, where $[m]$ denotes the set $\{1,\ldots,m\}$, and consider the (uncoloured) subgraphs $G_1$ given by the $q$ colour classes $1,\ldots, q$ and $G_2$ given by the $n$ colour classes $q+1,\ldots, q+n$. Note that we must have $G_1\to(H)_q$ or $G_2\to(H)_n$ since we could otherwise recolour $G$ with $n+q$ colours without a monochromatic copy of $H$, a contradiction. 
By equivalence in $n$ and $q$ colours we then have that $G_1\to(H')_q$ or $G_2\to(H')_n$ and hence the original colouring of $G$ admits a monochromatic copy of $H'$, so $G\to(H')_{n+q}$ as claimed. Similarly, every graph $G$ that is $(n+q)$-Ramsey for $H'$ needs to be $(n+q)$-Ramsey for $H$, which implies $\cM_{n+q}(H)=\cM_{n+q}(H').$

\thref{positiveLC} implies in particular that if two graphs $H$ and $H'$ are 2- {\em and} 3-equivalent, then they are $q$-equivalent for every $q\ge 2$. 
We wonder whether it is true that two graphs $H$ and $H'$ are 3-equivalent if they are 2-equivalent and whether this can be shown using ad-hoc methods. 

So far, we have investigated what we can deduce for $q\ge 3$ colours when we know that $H$ and $H'$ {\em are} 2-equivalent. What can we deduce when $H$ and $H'$ are {\em not} 2-equivalent? 
To examine this question let us return to the example of disjoint cliques from above. 
It is easy to see that $K_6$ is 2-Ramsey for $K_3$, yet fails to be Ramsey for the triangle and a disjoint edge, see e.g.~\cite{szz2010}. This shows that $K_3$ and $K_3+K_2$ are not $2$-equivalent. 
The following then implies that, in general, nothing can be deduced from non-2-equivalence. 
\begin{theorem}\thlab{thm:3equiv}
The graphs $K_3+K_2$ and $K_3$ are $q$-equivalent for all $q\ge 3$. 
\end{theorem}
In fact, there are infinitely many pairs of graphs that are not 2-equivalent, yet they are $q$-equivalent for some $q\ge 3$. 
To see this let us first mention how the criterion in~\cite{szz2010} generalises to more than two colours. 
For integers $q,k_1,\ldots, k_q\ge 2$ let $R(k_1,\ldots,k_q)$ denote the smallest integer $n$ such that any colouring of the edges of $K_n$ with colours $[q]$ contains a monochromatic copy of $K_{k_i}$ in colour $i$, for some $i\in [q]$. We write $R_q(k_1,k_2,\ldots,k_2)$ when $k_2=k_3=\ldots=k_q$.
\begin{theorem}\thlab{GeneralLowerBdd}
Let $k,t,q$ be integers such that $q\ge 2$ and $k>t\ge 2$. If $s<(R_q(k-t+1,k,\ldots,k)-q(k-t))/qt$ then $K_k$ and $K_k+s K_t$ are $q$-equivalent.
\end{theorem}
For $q=2$ and $t\le k-2$ this is Corollary~5.2~(ii) in~\cite{szz2010}, and the argument easily generalises to $q\ge 3$ colours. We provide the proof for completeness in Section~\ref{section4}. 
For $q=2$, \thref{GeneralLowerBdd}  
is known to be best possible up to a factor of roughly 2. Specifically, Fox, Grinshpun, Person, Szab\'o and the second author~\cite{fglps2014} show that for $k>t\ge 3$ the graphs $K_k$ and $K_k+s K_t$ are {\em not} $2$-equivalent if $s> (R(k-t+1,k)-1)/t$. This result implies the optimality of the equivalence of $K_k$ and $K_k+K_{k-1}$ in~\cite{bl2018} and the optimality up to a factor of roughly 2 in~\cite{szz2010} mentioned above. 
The consequence of this non-equivalence result in~\cite{fglps2014} and \thref{GeneralLowerBdd} is that, for given $k>t\ge 3$, the graphs $K_k$ and $K_k+sK_t$ are not 2-equivalent, but they are $q$-equivalent for some large enough $q$, if we take $s$ such that $(R(k-t+1,k)-1)/t < s< (R_q(k-t+1,k,\ldots,k)-q(k-t))/qt$.

The previous discussion shows that in general we cannot deduce non-$q$-equivalence for $q\ge 3$ from non-2-equivalence. However, all of the examples above that witness this phenomenon have at least one of $H$, $H'$ being disconnected. 
When both graphs $H$ and $H'$ are 3-connected or isomorphic to $K_3$ then $H$ and $H'$ are not $q$-equivalent for any $q\ge 2$, by \thref{dennisObs}. In fact, it remains an open question, first posed in~\cite{fglps2014}, whether there are two non-isomorphic connected graphs $H$ and $H'$ that are 2-equivalent. 
A theorem by Ne\v set\v ril and R\"odl~\cite{nr1976} implies that any graph that is $q$-equivalent to the clique $K_k$, for some $q\ge 2$, needs to contain $K_k$ as a subgraph. Fox, Grinshpun, Person, Szab\'o and the second author~\cite{fglps2014} show that $K_k$ is not $2$-equivalent to $K_k\cdot K_2$, the graph on $k+1$ vertices formed by adding a pendant edge to $K_k$. We lift this result to any number of colours. 
\begin{theorem}\thlab{thm:nonequiv}
For all $k, q\geq 3$, $K_k$ and $K_k\cdot K_2$ are not $q$-Ramsey equivalent.
\end{theorem}
Together with the result in~\cite{nr1976} this implies that, for all $q\ge 3$, $K_k$ is not $q$-equivalent to any connected graph other than $K_k$. 
We wonder whether one can prove in general that if two graphs $H$ and $H'$ are connected and not 2-equivalent, then they are not $q$-equivalent for any $q\ge 3$. In our proof of \thref{thm:nonequiv} the graph $K_k$ cannot be replaced by, say, $K_k$ missing an edge. 

The rest of the paper is organised as follows. In Section 2, we fix our notation and describe the method of signal senders. We also include the proof of \thref{dennisObs} there. In Section 3 we prove \thref{thm:containment,wow}. Section 4 contains the results related to Ramsey equivalence, that is we prove \thref{GeneralLowerBdd}, which we obtain as a corollary to a slightly more general result, as well as both \thref{thm:3equiv} and \thref{thm:nonequiv}. In the final section we discuss open problems.

\section{Preliminaries}\lab{sec:preliminaries}

{\bf Notation.} 
For a graph $G=(V,E)$ we write $V(G)$ and $E(G)$ for its vertex set and edge set, respectively, 
and we set $v(G)=|V(G)|$ and $e(G)=|E(G)|$. Throughout the paper we assume that $E(G)\se \binom{V(G)}{2}$ and that both $V$ and $E$ are finite. 
A graph $F$ is called a {\em subgraph of a graph $G$}, denoted by $F\se G$, 
if $V(F)\se V(G)$ and $E(F)\se E(G)$. 
Let $G$, $F$, and $H$ be graphs such that $F\se G$ and $V(G)\cap V(H) =\emptyset$. 
We write $G-F$ for the graph with vertex set $V(G)$ and edge set $E(G)\setminus E(F)$; 
and $G+H$ for the graph formed by the vertex-disjoint union of $G$ and $H$, i.e.~the graph with vertex set $V(G)\cup V(H)$ and edge set $E(G)\cup E(H)$. When $F$ or $H$ consist of a single edge $e$ we also write $G-e$ and $G+e$, respectively. 
For a subset $A\subseteq V(G)$ denote by $G[A]$ the {\em induced subgraph} on $A$, i.e.~the graph with vertex set $A$ and edge set consisting of all edges of $G$ with both endpoints in $A$. 
A subgraph $F$ of $G$ is called an {\em induced subgraph} if $F = G[V(F)]$. 
Given a path $P$ in a graph $G$, the {\em length of $P$} is the number of edges of $P$. 
For two subsets $A,B\subseteq V(G)$, we write $\dist_G(A,B)$ 
for the distance between $A$ and $B$, i.e.~the length
of a shortest path in $G$ with one endpoint in $A$
and the other endpoint in $B$. Given a subgraph $F\subseteq G$, 
we also write $\dist_G(A,F)$ for $\dist_G(A,V(F))$ and $\dist_G(A,e)$ if $F$ consists of a single edge $e$. 
A {\em $q$-colouring of a graph $G$} is a function $c$ that assigns colours to edges, where the set $S$ of colours has size $q$ and, unless specified otherwise, we assume that $S=[q]=\{1,\ldots,q\}$. 
We call a $q$-colouring {\em $H$-free} if there is no monochromatic copy of $H$.  

\bigskip

{\bf Signal senders.} 
For the proofs of \thref{thm:containment,dennisObs,thm:nonequiv} 
we use the idea of signal sender graphs which was first introduced by Burr, Erd\H{o}s and Lov\'asz~\cite{bel1976}. 
Let $H$ be a graph and $q\ge 2$ and $d\ge 0$ be integers. 
A \emph{negative (positive) signal sender $S=S^-(q, H,d)$ ($S=S^+(q, H,d)$)} 
is a graph $S$ containing distinguished edges $e,f\in E(S)$ 
such that 
\begin{enumerate}
\item[(S1)] $S\nto (H)_q$; 
\item[(S2)] in every $H$-free $q$-colouring of $E(S)$, 
	the edges $e$ and $f$ have different (the same) colours; and 
\item[(S3)] $\dist_S(e,f)\ge d$. 
\end{enumerate}

The edges $e$ and $f$ in the definition above are called 
\emph{signal edges} of $S$.
The following was proved by R\"odl and Siggers~\cite{rs2008}, generalising earlier proofs by 
Burr, Erd\H{o}s and Lov\'asz~\cite{bel1976} and by Burr, Ne\v{s}et\v{r}il and R\"odl~\cite{bnr1984}.

\begin{lemma}\thlab{lem:signal}
Let $H$ be 3-connected or $H=K_3$, and let $q,d\geq 2$ be integers. Then there exist 
negative and positive signal senders $S^-(q, H, d)$ and $S^+(q, H, d)$. 
\end{lemma}
In the proofs of \thref{thm:containment,dennisObs,thm:nonequiv} 
we construct graphs using several signal senders.  
Assume that $G$ is some graph and let $e_1,e_2\in E(G)$ be two disjoint edges. 
We say that we {\it join $e_1$ and $e_2$ by a signal sender $S(q,H,d)$} 
if we add a vertex disjoint copy $\tilde S$ of a signal sender $S(q,H,d)$ to $G$ and then identify the signal edges of $\tilde S$ with 
$e_1$ and $e_2$, respectively. 

\thref{dennisObs} is an easy consequence of the existence of signal senders, we prove it here to serve as a simple example of the method of signal senders. 
\begin{proof}[Proof of \thref{dennisObs}]
Without loss of generality let $H\nse H'$. 
Let $S = S^+(q, H', d)$ be a positive signal sender, where 
$d = v(H)+1$. If $S \to (H)_q$, then we are done since $S \nto (H')_q$ by (S1). So we may
assume that there is an $H$-free colouring $\phi :  E(S) \to [q]$. 
Now construct a graph $G$ as follows. Fix a copy $\tilde H$ of $H'$ and an edge $e$ that is vertex-disjoint
from $\tilde H$. Then, for every $f \in E(\tilde H)$ join $e$ and $f$ by a copy of the signal sender $S$ so that $e$ is always identified with the same signal edge of $S$. 
Then, $G \to (H')_q$. Indeed, for a $q$-colouring of $G$, there is a monochromatic copy of $H'$ in one of the copies of the signal sender $S$, or every edge in $\tilde H$ has the same colour as $e$, by (S2) and by construction of $G$. In either case, there is a monochromatic copy of $H'$. 

Moreover, $G \nto (H)_q$. Consider the colouring of $E(G)$ defined by colouring each copy of $S$ using $\phi$. Note that any two copies of $S$ intersect in the edge $e$ only (and at most one vertex in $\tilde H$). Since $e$ is always identified with the same signal edge in $S$ this colouring is well-defined. 
Now every copy of $H$ in $G$ is contained in a copy of the signal sender $S$ 
since $H\nse H'$, $H$ is 3-connected or $H\cong K_3$, and since $\dist_G(e,\tilde H) > v(H)$ by choice of $S$ and (S3). 
However, $\phi$ is $H$-free (on each copy of $S$), so none of these copies of $H$ is monochromatic.
\end{proof}


\section{Proof of \thref{thm:containment}}
\lab{sec:containment}
In order to prove \thref{thm:containment} we first establish the existence of certain gadget graphs. 
Let $F$ and $H$ be graphs and let $d,q\geq 2$ be integers. 
Let $G$ be a graph containing both an induced subgraph $\tilde F$ 
that is isomorphic to $F$ 
and an edge $e$ that is vertex-disjoint from $\tilde F$. 
$G$ is called an \emph{$(H, F, e,q,d)$-indicator} 
if $\dist_G(\tilde F,e)\ge d$ and 
the following hold for every $i,j\in [q]$: 

\begin{enumerate}
\item[(I1)] There exists an $H$-free $q$-colouring of $G$ 
	such that $\tilde F$ is monochromatic of colour $i$.
\item[(I2)] In every $H$-free $q$-colouring of $G$ 
	in which $\tilde F$ is monochromatic of colour $i$, $e$ has colour~$i$. 
\item[(I3)] If $f$ is any edge of $\tilde F$, then there exists an 
	$H$-free colouring of $G-f$ in which $\tilde F-f$ is monochromatic of colour~$i$ 
	and in which $e$ has colour~$j$.
\end{enumerate}

Note that it would be enough to say that the subgraphs and edges in the Properties~(I2) and~(I3) above should have the same or different colours respectively, without mentioning explicit colours $i$ and $j$ (since we can
swap the colours by symmetry). Nevertheless, we find it more convenient
to state the properties in the above manner, so that we do not need 
to repeat the argument of swapping colours over again.

The notion of indicators for $q=2$ was introduced by Burr, Faudree and Schelp~\cite{BFS} 
who established their existence in the case when $H$ is a clique and $F\not\supseteq H$, but with $d$ not being specified; see Lemma 3 in \cite{BFS}. 
We find the definition above to be a suitable generalisation for $q\ge 3$ to be able to prove existence while still being useful gadgets for the proof of \thref{thm:containment}. 

By definition it is necessary that $F$ does not contain a copy of $H$ for an $(H,F,e,q,d)$-indicator to exist. Under the assumption that $H$ is suitably connected this turns out to be sufficient. We need one more ingredient though which allows us to combine indicators (and signal senders) by identifying certain edges without creating new copies of $H$. 
We say that an $(H, F, e,q,d)$-indicator $G$ {\em \tentacle}
if there is a collection of subgraphs $\{T_f\se G \mid f\in E(\tilde F)\}$ 
such that 
\begin{enumerate}
\item[(T1)] $V(T_f)\cap V(\tilde F)=f$ and $f\in E(T_f)$ for all $f\in E(\tilde F)$,
\item[(T2)] $V(G)= \bigcup_{f\in E(\tilde F)} V(T_f)$ and $E(G) = \bigcup_{f\in E(\tilde F)} E(T_f)$, and
\item[(T3)] for all distinct $f_1,f_2\in E(\tilde F)$ and all $v\in V(T_{f_1})\cap V(T_{f_2})$ it holds that
$v\in V(\tilde F)$ or $\dist_G(v,\tilde F)\geq d$, 
\end{enumerate}
where $\tilde F$ is the fixed induced copy of $F$ in $G$.

\begin{lemma}\thlab{lem:indicator}
Let $H$ be 3-connected or $H=K_3$, let $F$ be a graph that does not contain a copy of $H$, let $e$ be an edge that is vertex-disjoint from $F$, and let $q,d\geq 2$ be integers. 
Then there exists an $(H, F, e,q,d)$-indicator $G$ that \tentacle . 
\end{lemma}

Similar to the convention for signal senders we say that, for given graphs $F\se G$ and an edge $e\in E(G)$ that is vertex-disjoint from $F$, we {\it join $F$ and $e$ by an $(H,F,e,q,d)$-indicator}  when we add a vertex-disjoint copy of an $(H,F,e',q,d)$-indicator $G'$ to $G$ and identify the copy of $F$ in $G'$ with $F\se G$ and identify the edge $e'$ in $G'$ with $e$ in $G$. 

The proof of \thref{lem:indicator} proceeds by induction on $e(F)$. When $F$ is a matching of two edges, however, we need gadget graphs with a stronger property than $(I3)$. We prove their existence first. 

\begin{lemma}\thlab{Case2}
Let $H, F, e, q, d$ be as in \thref{lem:indicator} and assume that $F=\{f_1,f_2\}$ is a matching. Then there exists an $(H, F, e,q,d)$-indicator $G_2$ with $\dist_{G_2}(f_1,f_2)\geq d$ that \tentacle, where instead of $(I3)$ we have that 
\begin{enumerate}
\item[$(I3')$] for $\ell\in\{1,2\}$ there exists an $H$-free colouring of $G_2$ 
	in which $f_\ell$ has colour $i$ and both, $e$ and $f_{3-\ell}$ have colour $j$.
\end{enumerate} 
\end{lemma}

\begin{proof}
We construct $G_2$ as follows. Start with a copy $\tilde F$ of $F$ and an edge $e$ that is vertex disjoint from $F$. By a slight abuse of notation we refer to $f_1$ and $f_2$ for the copies of the two edges of $F$. 
Let $\{e_1,e_2,\ldots,e_{q-1}\}$ be a matching of $q-1$ edges that are vertex-disjoint from $\tilde F$ and $e$. Let $H_1,H_2,\ldots,H_{q-1}$ be copies of $H$ that are vertex-disjoint from $f_1,f_2,e_1,e_2,\ldots,e_{q-1}$ and such that any two copies $H_i$ and $H_j$ intersect in one fixed edge which we identify with $e$. 
Furthermore, 
\begin{enumerate}
\item[(i)] join $f_1$ and $e_1$ by a negative signal sender $S_1=S^{-}(q,H,d)$ 
	and for every $2\leq k\leq q-1$ join $f_2$ and $e_k$ by a negative signal sender 
	$S_k=S^{-}(q,H,d)$; 
\item[(ii)] for every $1\leq k<\ell < q$ join $e_k$ and $e_{\ell}$ 
	by a negative signal sender $S_{k,\ell}=S^{-}(q,H,d)$;  
\item[(iii)] for every $1\leq k\leq q-1$ and every edge $g\in E(H_k-e)$ join $e_k$ and $g$ by a positive signal sender $S_{k,g}=S^{+}(q,H,d)$.  
\end{enumerate}
Note that the existence of the signal senders in (i)-(iii) is given by Lemma~\ref{lem:signal}.
Call the resulting graph $G_2$; an illustration can be found in Figure~\ref{fig:Case2} for the case that $q=4$.
\begin{figure}
	\begin{center}
	\includegraphics[scale=0.75,page=1]{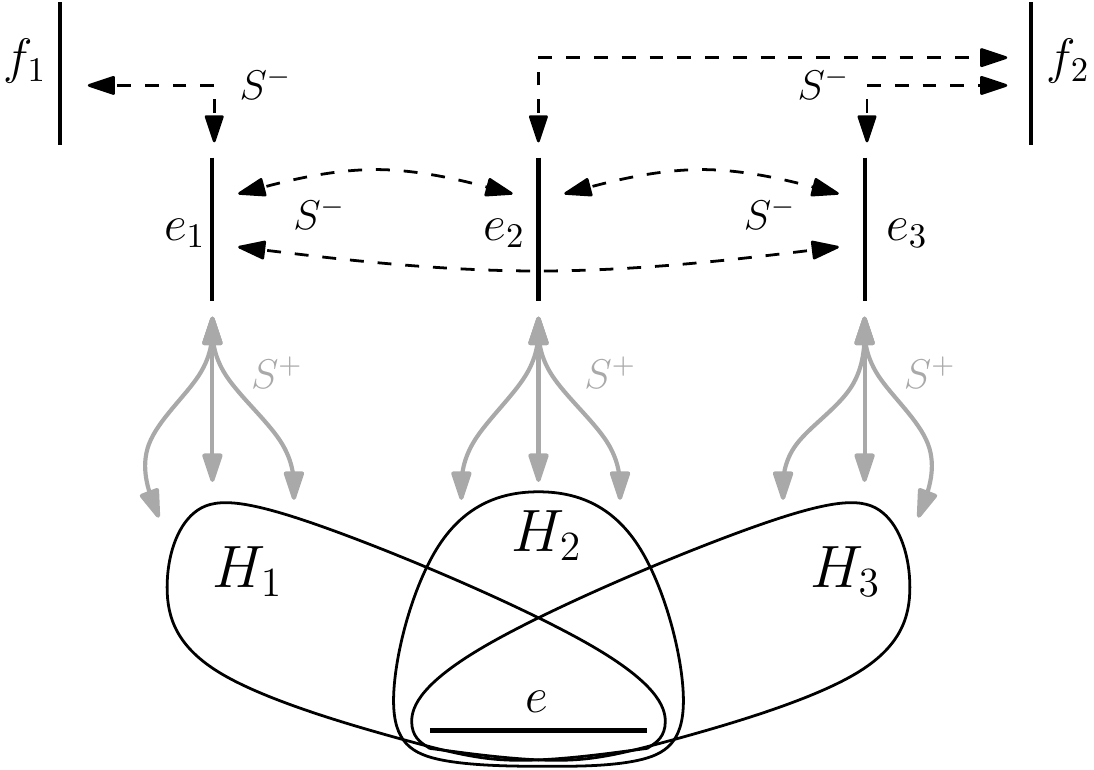}
	\end{center}
	\caption{Indicator for $q=4$ and $F=\{f_1,f_2\}$ being a matching.}
	\label{fig:Case2}
\end{figure}
It should be clear that $\dist_{G_2}(e,\tilde F)\geq d$ and $\dist_{G_2}(f_1,f_2)\geq d$. Thus, it remains to prove that $G_2$ satisfies Properties~$(I1)$, $(I2)$, $(I3')$ and Property \T. Without loss of generality we may assume that~$i=q$.

In the light of these properties, we first observe that every copy of $H$ in $G_2$ 
either is one of the subgraphs $H_k$ with $k\in [q-1]$ 
or is contained completely in one of the signal senders from (i)-(iii). 
Indeed, let a copy $H'$ of $H$ be given and assume first that $H'$ contains at least 
one vertex $v$ from a signal sender $S$ such that $v$ is not incident with one of the signal edges of $S$. 
Due to the fact that $H$ is 3-connected or $H=K_3$ and the fact that signal edges always 
have distance at least $d>v(H)$, it must hold that $H\subseteq S$. 
Assume then that $H'$ does not contain such a vertex. 
Then $H'$ must be contained in the union of all $H_k$ with $k\in [q-1]$.
As these subgraphs all intersect only in the edge $e$
and since $H$ is 3-connected or $H\cong K_3$, we must have $V(H')=V(H_k)$ for some $k\in [q-1]$.

For Property~$(I1)$, define a $q$-colouring of $G_2$ as follows. 
Colour the edges of $\tilde F$ and $e$ with colour $q$, and for every $k\in [q-1]$ colour the edges of $H_k-e$ and $e_k$ with colour $k$. Moreover, colour every signal sender from (i)-(iii) 
with an $H$-free $q$-colouring preserving the colours already chosen
for the signal edges. Note that this is possible by Properties~$(S1)$ and $(S2)$, 
because the signal senders may only intersect in their signal edges and the colours above 
have been chosen in such a way that the signal edges of negative/positive signal senders 
receive different/identical colours. The resulting $q$-colouring of $G_2$ is
$H$-free as it is $H$-free on every signal sender and on every subgraph $H_k$ with $k\in [q-1]$. 

For Property~$(I2)$, let $c:E(G_2)\rightarrow [q]$ be an $H$-free $q$-colouring of $G_2$ 
such that $\tilde F$ is monochromatic of colour $q$. 
Then $c(e_1)\neq c(f_1)=q$ and $c(e_k)\neq c(f_2)=q$ for every $k\in [q-1]$, by Property~$(S2)$ for the negative signal senders in (i). 
Similarly, by Property~$(S2)$ for the negative signal senders in (ii) we obtain that 
$c(e_k)\neq c(e_{\ell})$ for every $1\leq k<\ell \leq q-1$. 
Therefore, it must hold that $\{c(e_k):k\in [q-1]\}=[q-1]$. 
Applying $(S2)$ for the positive signal senders in (iii) we finally deduce that 
$H_k-e$ must be monochromatic in colour $c(e_k)$. 
Therefore, in order to prevent any copy $H_k$ of $H$ from becoming monochromatic 
we must have $c(e)\notin [q-1]$, i.e.~$c(e)=q$.

For Property~$(I3')$, let $f=f_{\ell}$, $\ell \in [2]$ be one of the two edges of $\tilde F$. We define a colouring $c:E(G_2)\rightarrow [q]$ as follows.
Set $c(f_{\ell})=c(e_{3-\ell})=q$, $c(e)=c(f_{3-\ell})=j$
and colour the edges $e_{\ell},e_3,e_4,\ldots,e_{q-1}$ 
with distinct colours from $[q-1]\setminus \{j\}$.
Colour the edges of $H_k-e$ with colour $c(e_k)$ for every $k\in [q-1]$. Finally, colour every signal sender from (i)-(iii) with an $H$-free colouring 
preserving the colours already chosen for the signal edges. 
Analogously to the verification of Property~$(I1)$ this is possible 
and it results in an $H$-free $q$-colouring of~$G_2$. Property~$(I3')$ follows.

For Property \T\ note that the choice $T_{f_1}=S_1$ and $T_{f_2}= G_2[V(G_2)\setminus V(S_1-e_1)]$ satisfies (T1)-(T3). 
\end{proof}

\begin{proof}[Proof of Lemma~\ref{lem:indicator}]
Without loss of generality we may assume that $d>v(H)$. We proceed by induction on $e(F)$. 

If $e(F)=1$ then let $G=S^+(q,H,d)$ be a positive signal sender, which exists by \thref{lem:signal}, and identify its signal edges with $e$ and $f$, where $f$ is the unique edge of $F$. Then Properties~$(I1)$ and $(I2)$ hold by Properties~$(S1)$ and $(S2)$ for positive signal senders. Property~$(I3)$ follows since $F-f$ has no edges, and by $(S1)$ again, after possibly swapping colours.
Property \T\ holds with $T_f=G$. 

Suppose now that $e(F)\ge2$. 
We construct $G$ as follows. Start with a copy $\tilde F$ of $F$ and an edge $e$ that is vertex disjoint from $\tilde F$. Let $e_1$ be an edge that is vertex-disjoint from $e$ and $\tilde F$, and let $f_1,f_2,\ldots,f_{e(F)}$ be
the edges of $F$ in any order. 
For clarity of presentation, we assume that the edges of $\tilde F$ are labelled  $f_1,f_2,\ldots,f_{e(F)}$ as well. 
Let $G_1$ be an $(H, F-f_1,e_1,q,d)$-indicator that \tentacle\ as given by induction, and let $G_2$ be an $(H,\{f_1,e_1\},e,q,d)$-indicator that \tentacle\ as given by \thref{Case2}. Now join $\tilde F-f_1$ and $e_1$ by $G_1$ and join $f_1$ and $e_1$ by $G_2$. 
An illustration can be found in Figure~\ref{figCase3}.
\begin{figure}[thb]
\begin{center}
\includegraphics[scale=0.7,page=2]{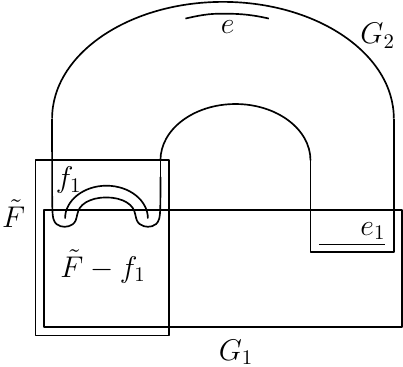}
\end{center}
\caption{Recursive construction of indicators.}
\label{figCase3}
\end{figure}

First observe that $\dist_G(e,\tilde F)\geq \min\{\dist_{G_2}(e,f_1),\dist_{G_2}(e,e_1)\}\ge d$ and $\dist_G(e_1,f_1)\geq d$. Furthermore, every copy of $H$ in $G$ must be either a subgraph of $G_1$ or of $G_2$. 
To see this, let $H'$ be a copy of $H$ in $G$. 
Assume first that $H'$ contains a vertex from $V(G_2)\setminus (e_1\cup f_1)$. Since $\dist_{G_2}(e_1,f_1)\geq d > v(H)$, $H'$ cannot use vertices from both $e_1$ and $f_1$, and thus, 
we conclude that $H'\subseteq G_2$ since either $H\cong K_3$ or $H$ is 3-connected. 
Assume then that $H'$ does not use vertices from $V(G_2)\setminus (e_1\cup f_1)$. 
This implies that $V(H')\se V(G_1)$ so we are done unless $f_1$ is an edge of $H'$ (note that by definition, the vertices of $f_1$ are vertices of $\tilde F-f_1$ and hence of $G_1$). 
Assume towards a contradiction that $f_1$ is an edge of $H'$. 
Furthermore, we may assume that $H'$ contains a vertex 
$v\in V(G_1)\setminus (V(\tilde F)\cup e_1)$ since $\tilde F$ does not contain a copy of $H$. 
Let $T_{f_2},\ldots,T_{f_{e(F)}}$ be the subgraphs of $G_1$ 
given by Property \T. Then $v\in V(T_g)$ for some $g\in E(F-f_1)$, by (T2). 
Furthermore, $g$ is unique since otherwise 
$\dist_{G}(v,f_1)\geq \min\{\dist_{G_1}(v,\tilde F),\dist_{G_2}(e_1,f_1)\}\geq d>v(H)$ by (T3), 
a contradiction since $H$ is connected. 
In fact, this shows that no vertex of $V(H')\setminus V(\tilde F)$ is contained in the intersection $V(T_{g_1})\cap V(T_{g_2})$ for distinct $g_1,g_2\in E(\tilde F-f_1)$. 
Now, $g$ cannot be incident with both endpoints of $f_1$, and $V(T_g)\cap V(\tilde F) = g$ by (T1). When $H\cong K_3$ this already implies that $v$ together with the vertices of $f_1$ cannot form a copy of $H$. 
When $H$ is 3-connected, we then find a vertex $w\notin V(T_g)$ 
which is incident with $f_1$ and
three internally vertex-disjoint $v$-$w$-paths in~$H'$.
Among these paths there is at least one path 
that does not contain a vertex from $g$, call this path $P$.
Since $v\in V(T_g)$ and $w\notin V(T_g)$
there must be an edge $e=xy$ on $P$ such that
$x\in V(T_g)$ and $y\notin V(T_g)$. But now, $x\notin g = V(T_g)\cap V(\tilde F)$, and thus we conclude that $x\notin V(\tilde F)$
and $e=xy \notin E(\tilde F)$. Using (T2) and $y\notin V(T_g)$
it follows that $e\in E(T_{g'})$ for some $g'\in E(\tilde F - f_1 - g)$. But then $x$ is a vertex in $V(H')\setminus V(\tilde F)$
which is contained in the intersection $V(T_g)\cap V(T_{g'})$
for distinct $g,g'\in E(\tilde F - f_1)$. We already explained that such a vertex does not exist, a contradiction.

For Property~$(I1)$, let $c_1$ be an $H$-free $q$-colouring of $G_1$with $\tilde F-f_1$ and $e_1$ having colour $q$, as provided by Properties~$(I1)$ and $(I2)$ for $G_1$. Analogously, let $c_2$ be an $H$-free $q$-colouring of $G_2$ with $\{e_1,f_1\}$ and $e$ having colour $q$. The combination of both colourings together is an $H$-free $q$-colouring $c$ of $G$, as every copy of $H$ is contained either in $G_1$ or in $G_2$. Moreover, $\tilde F$ is monochromatic in colour $q$, as claimed.

For Property~$(I2)$, let $c$ be an $H$-free $q$-colouring of $G$ such that $\tilde F$ is monochromatic of colour $q$. Then $c(e_1)=q$ by Property~$(I2)$ of  $G_1$. But then $\{e_1,f_1\}$ is monochromatic in colour $q$ which implies that 
$c(e)=q$ by Property~$(I2)$ of $G_2$. 

For Property~$(I3)$, let $f\in E(\tilde F)$. Assume first that $f=f_1$. 
As in $(I1)$ there exists an $H$-free $q$-colouring $c_1$ of $G_1$ 
with $\tilde F-f_1$ and $e_1$ having colour $q$. Moreover, 
using Property~$(I3)$ of $G_2$ we know that there is an $H$-free $q$-colouring $c_2$ 
of $G_2-f_1$ such that $c_2(e_1)=q$ and $c_2(e)=j$. 
The combination of both colourings is a $q$-colouring as desired, 
since every copy of $H$ is contained either in $G_1$ or in $G_2$. 
Now, assume that $f\neq f_1$. By Property~$(I3)$ of $G_1$ 
there is an $H$-free $q$-colouring $c_1$ of $G_1-f$ such that
$\tilde F-\{f_1,f\}$ is monochromatic in colour $q$ and with $e_1$ having colour $j$. 
By Property~$(I3')$ of $G_2$ there is an $H$-free $q$-colouring of $G_2$
such that $c_2(f_1)=q$ and $c_2(e)=c_2(e_1)=j$. 
The combination of both colourings is a $q$-colouring as desired for Property~$(I3)$.

For Property~\T, let $T_{f_2},\ldots,T_{f_{e(F)}}$ be the subgraphs for $f_2,\ldots,f_{e(F)}$ 
given by Property~\T\ of $G_1\subseteq G$.
Moreover, set $T_{f_1}=G_2$. Then (T1) 
holds for $G$, since (T1) holds for $G_1$
by induction and since $V(G_2)\cap V(\tilde F)=f_1$ and $f_1\in E(G_2)$.
Property~(T2) is given for $G$, 
since $V(G)=V(G_1)\cup V(G_2)=\bigcup_{f\in \tilde F-f_1}V(T_f)\cup V(T_{f_1})$ by Property~(T2) for $G_1$; and since $E(G) = E(G_1)\cup E(G_2) = \bigcup_{g\in E(\tilde F)}E(T_g)$.
For (T3), let $v\in V(T_{f_i})\cap V(T_{f_j})$
for some $i\neq j$ where $v\notin V(\tilde F)$.
If $i=1$ or $j=1$, then 
$v\in e_1$ and thus $\dist_G(v,\tilde F)\geq d$.
Otherwise, by (T3) for $G_1$ and 
the construction of $G$, we conclude that
$
\dist_G(v,\tilde F)\geq \min\{\dist_{G_1}(v,\tilde F-f_1), \dist_{G_2}(e_1,f_1)\}\geq d
$.
\end{proof}

%
%
\begin{proof}[Proof of Theorem \ref{thm:containment}.]
Let $H$, $q$, and $F$ be as in the theorem statement. By assumption, there exists an $H$-free $q$-colouring of $F$. Let $F_1,\ldots,F_q$ denote its colour classes. 
We construct a graph $G$ as follows. 
Let $\{r_1,\ldots,r_q,e_1,\ldots,e_q,f_1,\ldots,f_q\}$ be a matching that 
is vertex-disjoint from $F =F_1\cup\ldots\cup F_q$. 
Now join these matching edges and the edges of $F$ by signal senders and indicators as follows. Set $d=v(H)+1$. 
\begin{enumerate}
\item[(i)] For every $1\leq k<\ell\leq q$ join $r_k$ and $r_{\ell}$
	by a negative signal sender $S_{k,\ell} =S^{-}(q,H,d)$; 
\item[(ii)] for every $k\in [q]$ and every $g\in F_k$ join
	$r_k$ and $g$ by a positive signal sender 
	$S_{k,g}=S^{+}(q,H,d)$;  
\item[(iii)] for every $k\in [q]$ join $F_k$ and $e_k$ 
	by an $(H,F_k,e_k,q,d)$-indicator $I_k$
	that \tentacle ;  
\item[(iv)] for every $k\in [q]$ join $e_k$ and $f_k$ by a negative signal sender 
	$S_{k}^-=S^{-}(q,H,d)$; 
\item[(v)] for every $k\in [q-1]$ join $f_k$ and $f_{k+1}$ by a  positive signal sender 
	$S_{k}^+=S^{+}(q,H,d)$. 
\end{enumerate}
The existence of the signal senders and indicators in (i)-(v) follows from \thref{lem:signal,lem:indicator}. 
An illustration of the construction can be found in Figure~\ref{fig:containment}.
\begin{figure}
	\begin{center}
	\includegraphics[scale=0.75,page=4]{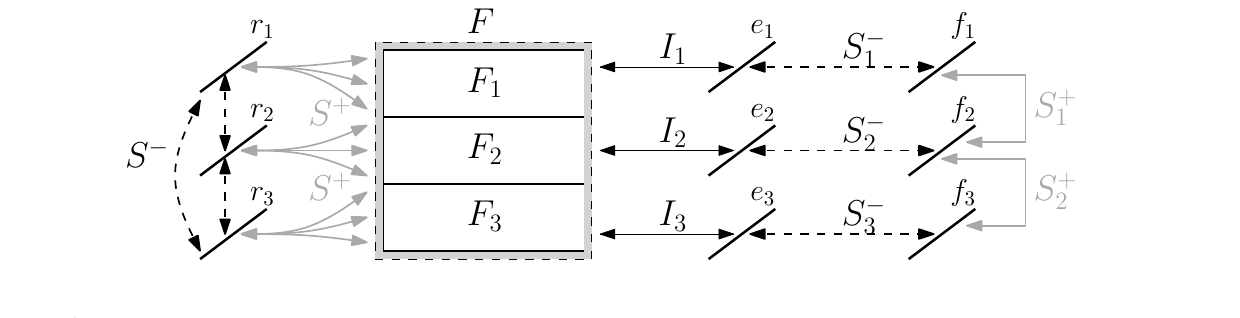}
	\end{center}
	\caption{Construction of $G$ for $q=3$.}
	\label{fig:containment}
\end{figure}

Similar to the constructions for Lemma~\ref{lem:indicator}, 
we first show that every copy of $H$ in $G$ is a subgraph of either $F$ or one of these signal senders or one of these indicators. Let $H'$ be a copy of $H$ in $G$. 
Assume first that there is a signal sender $S$ from (i), (ii), (iv) or (v), and a vertex $v\in V(H')\cap V(S)$ that is not incident to any of the signal edges of $S$. Then $H'\se S$, since the signal edges have distance at least $d$ in $S$ and since $H$ is $3$-connected or a triangle. 
So we may assume that $V(H')\se V(F)\cup \bigcup_{k\in[q]}V(I_k)$. 
If $V(H')\se V(F)$ then we are done. Thus, we may also assume that $H'$ contains a vertex $v$ from $V(I_k)\setminus V(F)$ for some $k\in [q]$. Let $\{T_g \mid g\in E(F_k)\}$ be the collection of subgraphs of $I_k$ 
given by Property~\T\ of $I_k$.
By (T2)
we know that $v\in V(T_g)\setminus V(F)$ for some $g\in E(F_k)$.
Moreover, the only edges of $G$ containing $v$ are contained in $I_k$ or in $S_k^-$, by construction. 
When $H'\cong K_3$ we immediately can deduce that
$V(H')\se V(I_k)$, since we already assumed
that $V(H')\cap (V(S_k^-)\setminus e_k) = \emptyset$.
Let $V(H')=\{v,x,y\}$ in this case. By Property~(T2)
we have $vx\in E(T_{g'})$ for some $g'\in E(F_k)$.
If $g\neq g'$ holds, then $v\in V(T_g)\cap V(T_{g'})$ and therefore
$\dist_{I_k}(v,F)\geq d = 4$ by Property~(T3),
which yields $H'\subseteq I_k$.
Hence, we may assume that $g=g'$ and $vx\in E(T_g)$,  
analogously we may assume that $vy\in E(T_g)$. 
If $xy\in E(I_k)$, then we are done.
So, we may also assume that $xy\in E(F)$.
Then by Property~(T1) we get 
$\{x,y\}\subseteq V(T_g)\cap V(F) = g$
and thus $xy=g\in E(T_g)$, again implying that $H'\subseteq I_k$.
Consider next the case when $H'$ is $3$-connected.
We already know that we may assume that 
$V(H')\se V(F)\cup \bigcup_{k\in[q]}V(I_k)$
and that there is a vertex $v$ from $V(I_k)\setminus V(F)$ 
for some $k\in [q]$. Again, by (T2) there is an edge $g\in E(F_k)$
with $v\in V(T_g)\setminus V(F)$. If $H'\subseteq T_g$ then we are done.
So, we may assume that $H'$ contains an edge
which does not belong to $E(T_g)$. By the 3-connectivity
of $H'$ we then find a $v$-$y$-path $P$ in $H'$ 
which does not use vertices from $g$ and such that
the edge $xy\in E(P')$ which is incident with the endpoint $y$ does not belong to $E(T_g)$. Let $P$ be a shortest such path. Then
$x\in V(T_g)$. Since $P$ does not use vertices of $g$
and since $V(T_g)\cap V(F)=g$ by (T1),
we have $x\notin V(F)$. In particular,
the only neighbours of $x$ in $V(F)$ need to belong to $g$,
and hence $y\notin V(F)$. It follows that $y\in V(I_k)\setminus V(F)$ and $xy\in E(I_k)$, since all the indicators intersect in $V(F)$ only. Hence, using (T1) and since $xy\notin E(T_g)$ by the choice of $P$, we find an edge $g'\in E(F_k-g)$ such that $xy\in E(T_{g'})$. But then $x\in V(T_g)\setminus V(F)$ is a vertex that lies in the intersection
$V(T_g)\cap V(T_{g'})$ for distinct $g,g'\in E(F_k)$.
Applying (T3) leads to $\dist_{I_k}(x,F)\geq d$
and hence $H'\subseteq I_k$.

We now prove that $(G-f)\nto (H)_q$ for every $f\in E(F)$. 
Without loss of generality let $f\in F_q$. 
We define a colouring $c:E(G-f)\to [q]$ as follows. Colour all edges of $F_q-f$ and $r_q$ with colour $q$, and for every $k\in [q-1]$ colour the edges of $F_k+\{e_k,r_k\}$ with colour $k$. Set $c(e_q)=1$ and $c(f_k)=q$ for every $k\in [q]$. Finally, colour every indicator from (iii) and every signal sender from (i), (ii), (iv) and (v) with an $H$-free $q$-colouring preserving the colours already chosen. For $I_q$ this is possible by Property~$(I3)$ and for all other indicators this is possible by Properties~$(I1)$ and $(I2)$. For the signal senders this is possible by Properties~$(S1)$ and $(S2)$, as the colours above have been chosen in such a way that the signal edges of negative/positive signal senders receive different/identical colours. 
We claim that $c$ is $H$-free. Indeed, any copy of $H$ is contained as a subgraph either in $F$ or in one of the indicators or signal senders as we have shown above. 
The colouring on each indicator and signal sender is $H$-free, and it is $H$-free on $F$ since each of $F_1,\ldots F_q$ receives a distinct colour, and each $F_i$ is $H$-free by assumption.

We next show that $G\to(H)_q$.  
Assume that there exists an $H$-free $q$-colouring $c$.
By Property~$(S2)$ of the negative signal senders
in (i), we find that $c(r_k)\neq c(r_{\ell})$ for all $k,\ell\in[q]$ with $k\neq \ell$. 
Without loss of generality let $c(r_k)=k$ for all $k\in [q]$.
By Property~$(S2)$ of the positive signal senders in (ii) it then follows
that $F_k$ needs to be monochromatic in colour $k$ for every $k\in [q]$.
Using Property~$(I2)$ of the indicators in (iii) we conclude that 
$c(e_k)=k$ must hold for every $k\in [q]$,
and applying Property~$(S2)$ of the negative signal senders in (iv) 
we then deduce $c(f_k)\neq k$ for every $k\in [q]$.
But then, using Property~$(S2)$ of the positive signal senders in (v), 
we obtain $c(f_1)=c(f_k)\neq k$ for every $k\in [q]$, a contradiction.

Finally, let $G'\subset G$ be a subgraph of $G$ that is $q$-Ramsey-minimal for $H$. 
Then $f\in E(G')$ for every $f\in E(F)$ since $(G-f)\nto(H)_q.$ Thus, $G'$ is a $q$-Ramsey-minimal graph for $H$ which contains $F$ as an induced subgraph. 

In order to obtain infinitely many such $q$-Ramsey-minimal graphs set $G_0=G'$ and obtain further such $q$-Ramsey-minimal graphs $G_i$ iteratively as follows. Let $F_i$ be the disjoint union of $v(G_{i-1})$ copies of $F$. Since $F_i$ is not $q$-Ramsey for $H$, we can repeat the above argument and thus create a 
$q$-Ramsey-minimal graph $G_i$ for $H$ which contains $F_i$ as an induced subgraph. Note that then $G_i$ also contains $F$ as an induced subgraph and $v(G_i)\geq v(F_i)>v(G_{i-1})$ holds. 
\end{proof}

\begin{proof}[Proof of \thref{wow}]
Suppose that for two graphs $H$ and $H'$ we have 
$\mathcal{M}_q(H)\se\mathcal{M}_q(H')$ and $\mathcal{M}_q(H')\nse\mathcal{M}_q(H)$. 
Let $G \in \mathcal{M}_q(H')\sm\mathcal{M}_q(H).$ 
If $G$ is $q$-Ramsey for $H$ then for some subgraph $G'$ of $G$ we have that $G'\in \cM_q(H)\se\cM_q(H')$ by assumption. 
If $G'=G$ this contradicts $G\not\in \cM_q(H)$, and if $G'$ is a proper subgraph of $G$ then this contradicts $G\in \cM_q(H')$ as $G$ is not minimal then. 
On the other hand, if $G$ is not $q$-Ramsey for $H$ then there exists a graph $G'$ such that $G\se G'\in\cM_q(H)\se\cM_q(H')$, by \thref{thm:containment} and assumption. Since $G\in\cM_q(H')$ by assumption it follows that $G=G'$, a contradiction to $G\not\in\cM_q(H)$. 
\end{proof}

\section{Ramsey equivalence results}\label{section4}

In this section we prove \thref{thm:3equiv,GeneralLowerBdd,thm:nonequiv}. 
We start with the proof of \thref{GeneralLowerBdd} which is a 
corollary of the following slightly more general statement. This multi-colour version is a straight-forward generalisation of the argument for 2 colours in \cite[Theorem 3.1]{szz2010}. \thref{GeneralLowerBdd}
follows by repeatedly applying this theorem to pairs $H_{i-1} = K_k+(i-2)K_t$ and $H_{i}=K_k+(i-1)K_t$ with $2\leq i\leq s +1$. 
\begin{theorem}\thlab{Thm3Pt1FromTibor}
Let $q\ge 2$, let $a_1\ge a_2\ge\ldots \ge a_s\ge 1$ and define $H_i:=K_{a_1}+\cdots + K_{a_i}$ for $1\le i \le s$. If 
$R_q(a_1-a_s+1,a_1,\ldots,a_1) >q(a_1+\ldots+a_{s-1}),$ then $H_s$ and $H_{s-1}$ are $q$-equivalent. 
\end{theorem}
\begin{proof}
It is clear that every graph $G$ that is $q$-Ramsey for $H_{s}$ is also $q$-Ramsey for $H_{s-1}$. Now let $G$ be a graph that is $q$-Ramsey for $H_{s-1}$. We need to show that $G$ is $q$-Ramsey for $H_s$. 
Suppose for a contradiction that $G\not\to (H_s)_q$ and let $c:E(G)\to[q]$ be a $q$-colouring of the edges of $G$ without a monochromatic copy of $H_s$. Without loss of generality, we may assume that there is a copy of $H_{s-1}$ in colour 1, and let $S_1$ be its vertex set. 
Since $c$ has no copy of $H_s$ in colour 1 the colouring restricted to $V(G)\sm S_1$ has no copy of $K_{a_s}$ in colour 1. 
Now, recursively for every colour $j=2,\ldots,q$, let $i_j$ be the largest index such that 
$V(G)\sm (S_1\cup\ldots\cup S_{j-1})$ contains a monochromatic copy of $H_{i_j}$ in colour $j$ (where we take $H_0$ to be the empty graph), and let $S_j$ be its vertex set. 
Since $c$ has no monochromatic copy of $H_s$ we have that $i_j<s$ for all $j\in[q]$. 
Now $c$ restricted to $V(G)\sm (S_1\cup\ldots\cup S_{q})$ does not contain a monochromatic copy of $K_{a_1}$,
since by the maximality of $i_j$ there is no copy of
$H_{i_j+1}= H_{i_j} + K_{a_{i_j+1}}$ in colour $j$
in $V(G)\setminus (S_1\cup\ldots\cup S_{j-1})$ 
and since $a_{i_j+1}\leq a_1$ for every $j\in [q]$.

As in the proof of Theorem 3.1 in \cite{szz2010} we now recolour some edges of $G$. 
We have that 
\begin{align*}
|S_1\cup\ldots S_q|=|V(H_{s-1})|+|V(H_{i_2})|+\ldots+|V(H_{i_q})|
	& \le q(a_1+\ldots+a_{s-1})\\
	& < R_q(a_1-a_s+1,a_1,\ldots,a_1).
\end{align*}
Hence, by the definition of the Ramsey number we can recolour the edges inside $S_1\cup\ldots\cup S_q$ without a monochromatic copy of $K_{a_1-a_s+1}$ in colour 1 and without a monochromatic copy of $K_{a_1}$ in colour $j$, for all $2\le j\le q$.  
All edges between $S_1\cup\ldots\cup S_q$ and $V\sm (S_1\cup\ldots\cup S_q)$ receive colour 1, and all remaining edges retain their original colour. 
It is now easy to see that there is no monochromatic copy of $K_{a_1}$ which is a contradiction to $G\to(H_{s-1})_q$. 
\end{proof}

It turns out that \thref{GeneralLowerBdd} already implies \thref{thm:3equiv} for $q\ge 4$. We need two more ingredients for the case~$q=3$.

\begin{observation}\thlab{obsii}
Let $G$ be a graph such that $G\to (K_3)_3$, and let $c$ be a 3-colouring of the edges of $G$. If there is a monochromatic copy of $K_3$ in every colour, then there is a monochromatic copy of $K_3+K_2$. 
\end{observation}
\begin{proof}
We first note that $\chi(G)\ge R_3(3) = 17$, where $\chi(G)$ is the chromatic number of $G$, see, e.g., Theorem 1 in \cite{l1972}. 
Let $V_0$ be the set of vertices belonging to the three monochromatic triangles, each of a different colour, which exist by assumption. 
Then $G[V(G)\backslash V_0]$ contains an edge as otherwise $\chi(G)\leq\chi (G[V_0])+1\leq 10$. This edge then forms  a monochromatic copy of $K_3+K_2$ along with one of the three monochromatic triangles.
\end{proof}

The next theorem was proved by Bodkin and Szab\'o (see \cite{b2015}, and \cite{bl2018} for a proof).
\begin{theorem}[Theorem 2 in \cite{bl2018}] \thlab{K6Distinguisher}
If $G\to (K_3)_2$ and $G\not\to (K_3+K_2)_2$ then $K_6\se G$. 
\end{theorem}

\begin{proof}[Proof of \thref{thm:3equiv}]
For $k=3$ and $t=2$, \thref{GeneralLowerBdd} implies that $K_3$ and $K_3+K_2$ are $q$-equivalent if $R_q(2,3,\ldots,3) = R_{q-1}(3) > 3q$. This inequality follows for $q\ge 4$ easily by induction on $q$ with the induction start given by the fact $R_3(3)=17>4\cdot 3$.

It remains to prove that $K_3$ and $K_3+K_2$ are $3$-equivalent. 
Clearly, any graph which is $3$-Ramsey for $K_3+K_2$ is also $3$-Ramsey for $K_3$. 
Let now $G$ be a graph that is $3$-Ramsey for $K_3$ and let $c$ be a 3-colouring of $G$ using colours red, blue, and yellow. Let $R$, $B$, and $Y$ denote the subgraphs formed by the red, blue, and yellow edges, respectively. We need to show that we can find a copy of $K_3+K_2$ in one of $R$, $B$, or $Y$. 

Suppose first that none of the subgraphs of $G$ formed by the union of any two of $R, B, Y$ is a $2$-Ramsey graph for $K_3$. Then the subgraph $R\cup B$ can be recoloured red-blue without monochromatic copies of $K_3$. 
Hence there must exist a (yellow) copy of $K_3$ in $Y$, since $G\to (K_3)_3$. 
Similarly we argue that there is also both a blue and a red copy of $K_3$ in $G$. 
We are then done by \thref{obsii}. 

Suppose now that without loss of generality $R\cup B$ is $2$-Ramsey for $K_3$.  
Then by \thref{K6Distinguisher} either there is a copy of $K_3+K_2$ in $R$ or in $B$ (and we are done); or $K_6$ is a subgraph of $R\cup B$, say on vertex set $S$. Now we find either a red or a blue copy of $K_3+K_2$ in $S$; or both a red and a blue copy of $K_3$ on $S$. 

We claim that $G$ contains a further (not necessarily monochromatic) copy of $K_3$ in $V(G)\sm S$. Suppose not. Then we recolour $G$ as follows. 
 Let $v\in S$ and colour the edges of $G[V(S)\setminus \{v\}]$ with red and blue without a monochromatic copy of $K_3$ (i.e.~a red and a blue $C_5$). Colour all edges incident to $v$ in $S$  yellow and colour all edges in $V(G)\sm S$ blue. Finally, colour all edges between $V(S)\setminus \{v\}$ and $V(G)\sm S$ yellow and all those between $v$ and $V(G)\sm S$ red. Unless there is a triangle in $V(G)\sm S$ this colouring does not contain a monochromatic copy of $K_3$, a contradiction to $G\to (K_3)_3$. 
Let $T$ be this triangle in $G-S$. If any of the edges of $T$ is red or blue, then this edge forms a monochromatic copy of $K_3+K_2$ with one of the monochromatic triangles in $S$. Otherwise, all edges of $T$ are yellow, and we are done again by \thref{obsii}.
\end{proof}

We now turn to the proof of \thref{thm:nonequiv}. To show the non-equivalence of two graphs $H$ and $H'$ we need to construct a graph that is $q$-Ramsey for one of the graphs, say for $H$, and not $q$-Ramsey for $H'$. Recall that the signal senders in Section~\ref{sec:preliminaries} provide us with graphs that can enforce certain predefined colour patterns. We now introduce suitable colour patterns. 
Following notation of~\cite{fglps2016}, we call a graph $F$ on $n$ vertices 
$(n,r,k)$-critical if $K_{k+1}\not\se F$ and every subset $S\se V(F)$ of size $|S|\ge n/r$ satisfies $K_k\se F[S]$. 
A sequence of pairwise edge-disjoint graphs $F_1,\ldots, F_r$ on the same vertex set $V$ is called a {\em colour pattern} on $V$.

\begin{lemma}[Lemmas 4.2 and 4.4 in \cite{fglps2016}]\thlab{nrkcritical}
Let $k\ge 2$, $r\ge 3$ be integers. Then there exists a colour pattern $F_1,\ldots,F_r$ on vertex set $[n]$, for some $n$, such that each $F_i$ is $(n, r, k)$-critical.
\end{lemma}
\begin{remark}
The results in \cite{fglps2016} include bounds on $n$ in terms of $r$, which is unnecessary for our purpose. Without these bounds, the lemma can actually be proved by a now standard application of the probabilistic method. 
\end{remark}

Next we state a lemma which captures the effect of repeated application of the pigeonhole principle in a coloured bipartite graph. Its proof is a straight-forward generalisation of the proof of Lemma 2.6~(a) in~\cite{fglps2014}. 
\begin{lemma}\thlab{focusing}
Let $G=(A\cup B, E)$ be a complete bipartite graph with a $q$-colouring $c:E\to[q]$ of its edges. Then there exists a subset $B'\se B$ with $|B'|\ge |B|/q^{|A|}$ such that for every vertex $a\in A$ the set of edges from $a$ to $B'$ is monochromatic. 
\end{lemma}

We are ready to prove \thref{thm:nonequiv}. 

\begin{proof}[Proof of \thref{thm:nonequiv}] 
Fix $k\geq 3$. 
The proof proceeds by induction on $q$. 
For $q=2$, there exists a graph $G_2$ that satisfies $G_2\to K_k$ and $G_2\nto K_k\cdot K_2$, by~\cite{fglps2014}. 
So assume that $q\ge 2$ and let $G_q$ be a graph such that is $G_q\to (K_k)_q$ and $G_q \nto (K_k\cdot K_2)_q$. We construct a graph $G_{q+1}$ with the properties $G_{q+1} \to (K_k)_{q+1}$ and $G_{q+1} \nto (K_k\cdot K_2)_{q+1}$. 

Let $r=q^{|V(G_q)|+qk^2}+1$, and let 
$F=F_1\cup\ldots\cup F_q$ be a colour pattern such that each $F_i$ is $(n,r,k-1)$-critical for some $n$. The existence of $F$ follows from \thref{nrkcritical}. (Note that we only use $q$ pairwise edge-disjoint graphs, where the lemma in fact provides $r$ such graphs $F_i$.) 
We construct $G_{q+1}$ as follows. 
Let $\tilde G_q$ be a copy of $G_q$, say on vertex set $V_0$. 
Let $V_1,\ldots, V_{k-2}$ be pairwise vertex disjoint sets of size $n=|V(F)|$ that are disjoint from $V_0$. Let $\{e_1,\ldots,e_q\}$ be a matching of size $q$, (vertex-) disjoint from $V_0\cup\ldots\cup V_{k-2}$. 
For each $1\le j \le k-2$ let $F^{(j)}=F_1^{(j)}\cup\ldots\cup F_q^{(j)}$ be a copy of $F$ on vertex set $V_j$. 
Additionally, add all edges between $V_i$ and $V_j$ for all $0\le i<j\le k-2$. 
Finally, we join edges by signal senders in the following way. 
For all $1\le i < j\le q$, join $e_i$ and $e_j$ by a negative signal sender 
$S^{-}=S^{-}(q+1,K_k,k)$. 
And for all $1\le i \le q$ and every edge $e\in F_i^{(1)}\cup\ldots\cup F_i^{(k-2)}$ join $e$ and $e_i$ by a positive signal sender $S^{+}=S^{+}(q+1,K_k,k)$. 
Both signal senders $S^{-}$ and $S^{+}$ exist by \thref{lem:signal}. 
The resulting graph is $G_{q+1}$, an illustration can be found in Figure~\ref{fig:nonEquiv}. 
\begin{figure}
	\begin{center}
	\includegraphics[width=0.7\textwidth]{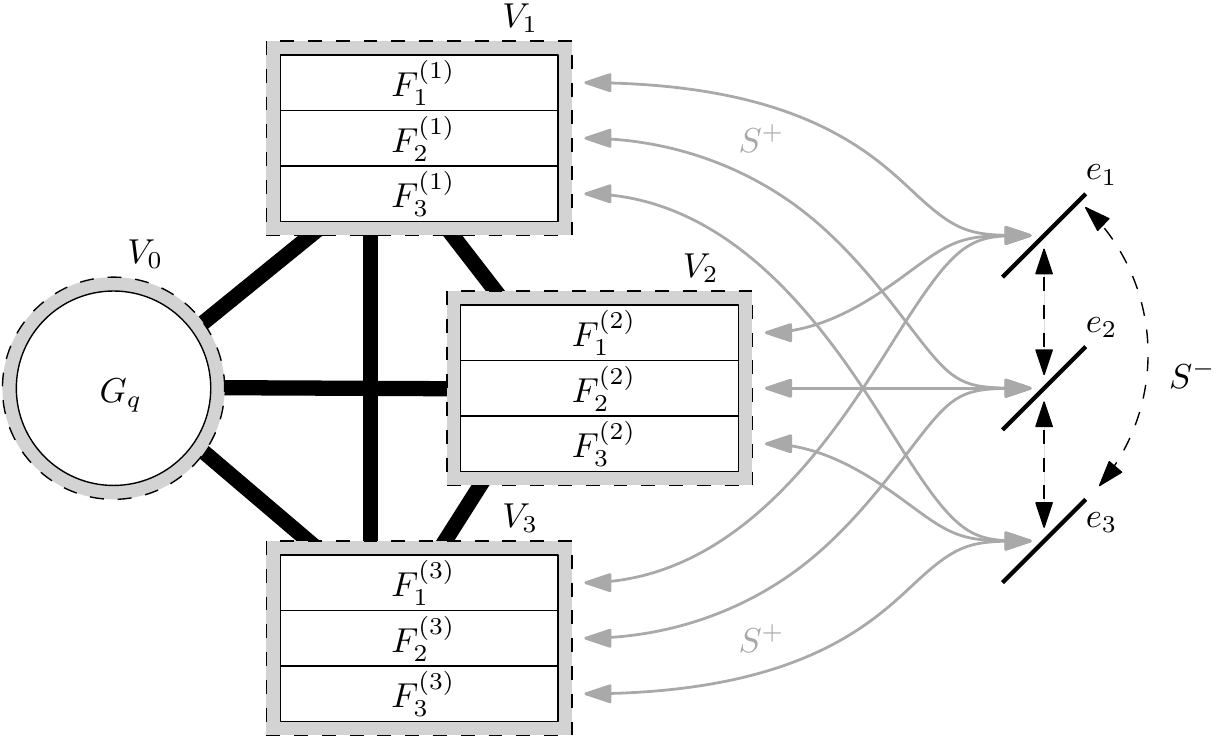}
	\end{center}
	\caption{An illustration of the graph $G_{q+1}$ when $q=3$ and $k=5$. Bold edges indicate complete bipartite graphs.}
	\label{fig:nonEquiv}
\end{figure}
\begin{claim}
$G_{q+1}\nto (K_k\cdot K_2)_{q+1}$. 
\end{claim}
\begin{proof}
Consider the following $(q+1)$-colouring of the edges of $G_{q+1}$. 
By inductive hypothesis of $G_q$, there exists a $(K_k\cdot K_2)$-free 
colouring $c_0:E(\tilde G)\to [q]$ of the edges in $V_0$.
For all $1\le i \le q$, colour the edges of $F_i^{(1)}\cup\ldots\cup F_i^{(k-2)}$ and the edge $e_i$ in colour $i$. 
Colour all edges between any $V_i$ and $V_j$, $0\le i<j\le k-2$, with colour $q+1$. 
Note that all pairs of edges that are joined by copies of $S^+$ have the same colour. 
There exists a $K_k$-free $(q+1)$-colouring $c^+$ of $S^+$ by Property~$(S1)$, 
and by Property~$(S2)$ both signal edges have the same colour in $c^+$. Extend the partial colouring of $G_{q+1}$ to every copy of $S^+$ using $c^+$ (possibly permuting the colours so that the colouring agrees on the already coloured signal edges). Similarly, any two edges that are joined by copies of $S^-$ received distinct colours (edge $e_i$ received colour $i$ for $i\in [q]$); and there exists a $K_k$-free $(q+1)$-colouring $c^-$ of $S^-$ by Property~$(S1)$ in which the two signal edges have distinct colours. Extend the partial colouring further to every copy of $S^-$ using $c^-$, again permuting colours when needed. 
 
We claim that this gives a $(K_k\cdot K_2)$-free $(q+1)$-colouring of $G_{q+1}$. 
First note that any copy of $K_k$ is either contained in $V_0\cup\ldots\cup V_{k-2}$, or is contained in one of the copies of a signal sender. This follows since the intersection of the vertex set of every copy of $S^+$ (or $S^-$) and $V_0\cup\ldots\cup V_{k-2}\cup \bigcup_{k\in [q]} e_k$ contains at most the two signal edges of the signal sender, and since the distance between those two edges is at least $k$ in $G_{q+1}$. 
The colouring is $K_k$-free on every copy of a signal sender by the choice of the colourings $c^+$ and $c^-$. Next, note that the edges of colour $(q+1)$ in $V_0\cup\ldots\cup V_{k-2}$ form a (complete) $(k-1)$-partite graph as no edge inside $V_i$, $0\le i\le k-2$, has colour $q+1$. Thus there is no monochromatic copy of $K_k$ in colour $q+1$ in $G_{q+1}.$  
Furthermore, for every $1\le i\le q$, the graph formed by edges of colour $i$ on vertex set $V_1\cup\ldots\cup V_{k-2}$ is isomorphic to the vertex-disjoint union of copies of $F_i$ which is $(n,r,k-1)$-critical and thus $K_k$-free. 
It follows that the only monochromatic copies of $K_k$ are contained in $V_0$. The colouring on $V_0$ only uses the colours $[q]$, whereas all edges between $V_0$ and $V(G_{q+1})\sm V_0$ have colour $q+1$. Furthermore, the colouring on $V_0$ is $K_k\cdot K_2$-free, by inductive assumption. Therefore, if there is a monochromatic copy of $K_k$, then it must be contained in $V_0$, and then there is no pendant edge to that copy of the same colour. 
\end{proof}
\begin{claim}
$G_{q+1}\to (K_k)_{q+1}$. 
\end{claim}
\begin{proof}
Let $c:E(G_{q+1})\to[q+1]$ be a $(q+1)$-colouring and suppose that there is no monochromatic copy of $K_k$ in this colouring. 
Then $c$ is $K_k$-free on every copy of $S^-$. Thus the two edges $e_i$ and $e_j$ receive different colours for all $1\le i<j\le q$, by Property~$(S2)$ of a negative signal sender. After permuting colours we may henceforth assume that the edge $e_i$ has colour~$i$ for $1\le i\le q$. 
Furthermore, $c$ is $K_k$-free on every copy of $S^+$ which joins $e$ and $e_i$, for each~$i\in[q]$ and $e\in F_i^{(1)}\cup\ldots\cup F_i^{(k-2)}$. This implies that the graph 
\begin{align}\lab{aux665}
& F_i^{(1)}\cup\ldots\cup F_i^{(k-2)} \text{ is monochromatic of colour $i$ for every }i\in [q],
\end{align}
by Property~$(S2)$ for positive signal senders. 

We now apply \thref{focusing} to the bipartite graph between $V_0$ and $V_1$ and deduce that there is a set $V_1'\se V_1$ with $|V_1'|\ge |V_1|/q^{|V_0|}$ such that 
for every vertex $v\in V_0$ the set of edges from $v$ to $V_1'$ is monochromatic. 
Now, $|V_1|/q^{|V_0|}\ge |V_1|/r$ by choice of $r$. Hence, for every~$i\in[q]$ there is a monochromatic copy of $K_{k-1}$ in colour $i$ in $V_1'$, say on vertex set $W_1^{(i)}$, since $F_1^{(i)}$ is $(n,r,k-1)$-critical and monochromatic of colour $i$, by~\eqref{aux665}. 
Let $W_1= \bigcup_{i\in [q]} W_1^{(i)}$ and note that $|W_1|\le qk$. 
If there exists a vertex $v\in V_0$ such that all the edges from $v$ to $W_1\se V_1'$ have colour $i$ for some $i\in[q]$ then the vertices $W_1^{(i)}\cup\{v\}$ form a monochromatic copy of $K_k$ in colour $i$ and we are done. We may thus assume that all edges between $V_0$ and $W_1$ have colour $q+1$. 

Iteratively assume that we have defined $W_1,\ldots, W_{\ell}$ for some $\ell=1,\ldots,k-3$, such that for every $i,j\in [\ell]$ with $i\ne j$ we have that $W_i\se V_i$ of size $|W_i|\le qk$, $W_i$ contains a monochromatic copy of $K_k$ in every colour $j\in[q]$, all edges between $V_0$ and $\bigcup_{i\in[\ell]} W_i$ have colour $q+1$, and all edges between $W_i$ and $W_j$ have colour $q+1$. 
We then obtain $W_{\ell+1}$ in $V_{\ell+1}$ by repeating the argument above where $V_0$ is replaced by $V_0\cup W_1\cup \ldots \cup W_{\ell}$. Note that this set has size at most $|V_0|+qk^2$. Thus the subset $V_{\ell+1}'\se V_{\ell+1}$ that we obtain by application of \thref{focusing} has size at least $|V_{\ell+1}|/q^{|V_0|+qk^2}\ge |V_{\ell+1}|/r$ by choice of $r$. The rest of the argument is analogous. 

Thus either we find a monochromatic copy of $K_k$ in one of the colours $1,\ldots,q$; or we obtain sets $W_1,\ldots, W_{k-2}$ that form a complete $(k-2)$-partite graph in colour $q+1$ and such that all edges between $V_0$ and 
$\bigcup_{i\in[k-2]} W_i$ are present and have colour $q+1$. If any of the edges in $V_0$ has colour $q+1$, then this edge together with one vertex from each $W_i$, $i\in [k-2]$, forms a monochromatic copy of $K_k$ in colour $q+1$, and we are done again. Otherwise, no edge in $V_0$ has colour $q+1$. But the graph on $V_0$ is isomorphic to $G_q$ which means that in any $q$-colouring of the edges in $V_0$ there is a monochromatic copy of $K_k$ in at least one of the colours. 
\end{proof}
This finishes the proof of \thref{thm:nonequiv}.
\end{proof}

\section{Concluding remarks}

\noindent
{\bf Minimal minimum degree of minimal Ramsey graphs.\\} 
We have proved that $K_k$ and $K_k\cdot K_2$ are not $q$-equivalent for any $q\ge 3$. The proof proceeds by induction on $q$ with the base case given by the non-equivalence in two colours from~\cite{fglps2014}. The {\em $2$-distinguishing} graph $G_2$ constructed in~\cite{fglps2014} actually has a stronger property, namely that $G\nto (K_k\cdot K_2)_2$ and every $(K_k\cdot K_2)$-free colouring of $G_2$ has a {\em fixed} copy of $K_k$ being monochromatic. 
This stronger property was used there to construct a graph $G'$ that is 2-minimal for $K_k\cdot K_2$ and that contains a vertex of degree $k-1$, i.e.~$s_2(K_k\cdot K_2)\le k-1$. The classical paper by Burr, Erd\H{o}s, and Lov\'asz contains the proof of $s_2(K_k) = (k-1)^2$, i.e.~adding a pendant edge to $K_k$ changes the behaviour of $s_2(\cdot)$ drastically. 
\begin{problem}
Determine $s_q(K_k\cdot K_2)$ for $q\ge 3$. 
Specifically, is it true that $s_q(K_k\cdot K_2) \le s_q(K_k)$, and if so, how small is the ratio $s_q(K_k\cdot K_2)/s_q(K_k)$?
\end{problem}  
It is known that $s_q(K_k) = O(q^2(\ln q)^{8(k-1)^2})$ for $k\ge 4$ where the implicit constant is independent of $q$ \cite{fglps2016}. For fixed $k$, this bound is tight up to a factor that is polylogarithmic in $q$. Furthermore, $s_q(K_3) = \Theta(q^2\log q)$ \cite{gw2017}. 

The construction of $G_2$ in~\cite{fglps2014} does not generalise in a straight-forward manner to more than 2 colours. The $q$-distinguishing graph $G_q$, $q\ge 3$, from the proof of \thref{thm:nonequiv} contains signal senders and thus does not have the stronger property of having a fixed copy of $K_k$ that is monochromatic in every $(K_k\cdot K_2)$-free $q$-colouring of $G_q$ as $G_2$. In particular, our graphs $G_q$ cannot be used (per se) for constructions showing upper bounds on $s_q(K_k\cdot K_2)$.

\medskip
\noindent
{\bf From $2$-(non)-equivalence to multicolour-(non)-equivalence.\\} 
We have seen in the introduction that $2$-equivalence of $H$ and $H'$ implies $q$-equivalence for every even $q$. More generally,  \thref{positiveLC} implies that two graphs are $q$-equivalent for every $q\ge 3$ if they are known to be $2$-equivalent {\em and} $3$-equivalent. We reiterate our question from the introduction here. 
\begin{question}
Is it true that any two 2-equivalent graphs $H$ and $H'$ are also 3-equivalent? 
\end{question}
Or are there two graphs $H$ and $H'$ that are, say, $100$-equivalent but not $101$-equivalent? 
We have also said in the introduction that in general one cannot deduce that $H$ and $H'$ are not $q$-equivalent for $q\ge 3$ from the mere fact that they are not $2$-equivalent. All examples had $H$ or $H'$ being disconnected. Is this a coincidence? 
\begin{question}
Let $H$ and $H'$ be both connected graphs that are $3$-equivalent. Is it true that they are 2-equivalent as well? 
\end{question}
This question may have an affirmative answer for the trivial reason that there are no two connected non-isomorphic graphs $H$ and $H'$ that are $q$-equivalent for any $q\ge 2$. This question was first posed in~\cite{fglps2014} for two colours, and we extend it here to any number of colours. 
\begin{question}
For given $q\ge 2$, are there two non-isomorphic connected graphs $H$ and $H'$ that are $q$-equivalent?
\end{question}
Since $K_k$ is not $q$-equivalent to any other connected graph 
(see the discussion preceding \thref{thm:nonequiv})
and since any two 3-connected graphs are not $q$-equivalent for any $q\ge 2$ by \thref{dennisObs} it is generally believed that the answer to this question is no. 

\medskip
\noindent
{\bf Adding a connected graph to a clique.\\} 
We have seen that $K_k$ is Ramsey equivalent to $K_k +H$ where $H$ is a collection of vertex-disjoint cliques. What other graphs $H$ have that property? Here we concentrate on the 2-colour case to highlight how little is known. Of course, all the following questions have natural analogues in the multicolour setting. 
We know that $K_k$ and $K_k+K_k$ are not Ramsey equivalent (since the clique on $R_2(k)$ vertices is a distinguisher) and that $K_k$ and $K_k+ K_{k-1}$ are Ramsey equivalent. The following three questions are, of course, related, we find each of them interesting. 
\begin{question}
\begin{itemize}

\item
What is the largest value of $t=t(k)$  such that there is a connected graph $H$ on $t$ vertices so that $K_k$ and $K_k+H$ are Ramsey equivalent?  

\item 
What is the largest value of $t=t(k)$  such that $K_k$ and $K_k+S_t$ are Ramsey equivalent, where by $S_t$ we denote the star with $t$ vertices (in alignment with the previous question)?  

\item 
What is the largest value of $t=t(k)$  such that $K_k$ and $K_k+P_t$ are Ramsey equivalent, where by $P_t$ we denote the path with $t$ vertices?  

\end{itemize}
\end{question}

The second question is from \cite{fglps2014}. Note that the equivalence of $K_k$ and $K_k+K_{k-1}$ implies that the answer to these questions is at least $k-1$. Moreover, it is easy to obtain an upper bound of roughly $R(k)$, i.e.~exponential in $k$. To the best of our knowledge nothing better is known.
Specifically, we wonder whether $K_k$ and $K_k+S_k$ are Ramsey-equivalent. If the answer is affirmative then this may shed light on whether $K_k +K_{k-1}\cdot K_2$ and $K_k$ are Ramsey equivalent. Slightly more ambitious is the following. 
\begin{problem}
Are $K_k$ and $K_k+K_k^-$ Ramsey equivalent, where $K_k^-$ denotes the clique on $k$ vertices with one edge deleted?
\end{problem}
An affirmative answer would imply that $R(K_k^-)<R(K_k)$, an inequality conjectured to be true, but only known for $k\le 6$, see e.g.~\cite{BLS}.

\bigskip 
\noindent
{\bf Acknowledgement.} 
This research was started at Monash University,
during a research stay of the first and third author 
who would like to express their gratitude for 
hospitality and for providing a perfect working environment.  
The third author would also like to thank Vojtech R\"odl for a helpful discussion.

\bibliographystyle{pagendsort}
\bibliography{references}

\end{document}